\documentclass[a4,12pt]{amsart}
\usepackage{amssymb,amstext,amsmath,amscd,amsthm,amsfonts,enumerate,latexsym}
\usepackage{color}
\usepackage[dvipdfmx]{graphicx}
\usepackage[all]{xy}
\usepackage{stmaryrd} 

\usepackage[dvipdfmx]{hyperref}

\usepackage{mathptmx}
\theoremstyle{plain}
\newtheorem{thm}{Theorem}[section]

\newtheorem*{thm*}{Theorem}
\newtheorem*{cor*}{Corollary}

\newtheorem{prop}[thm]{Proposition}

\newtheorem{lem}[thm]{Lemma}
\newtheorem{cor}[thm]{Corollary}

\newtheorem*{claim*}{Claim}

\theoremstyle{definition}

\newtheorem{ex}[thm]{Example}

\newtheorem{ques}[thm]{Question}

\newtheorem{setup}[thm]{Setup}

\theoremstyle{remark}
\newtheorem{rem}[thm]{Remark}

\numberwithin{equation}{thm}

\newtheorem*{ac}{Acknowledgments}

\def\Z{\mathbb{Z}}

\def\Ext{\operatorname{Ext}}

\def\Im{\operatorname{Im}}

\def\bbZ{\mathbb{Z}}
\def\bbN{\mathbb{N}}

\def\Hom{\operatorname{Hom}}
\def\RHom{\mathrm{{\bf R}Hom}}

\def\Coker{\mathrm{Coker}}

\def\m{\mathfrak m}
\def\n{\mathfrak n}

\def\p{\mathfrak p}

\def\K{\mathrm{K}}

\newcommand{\rma}{\mathrm{a}}

\newcommand{\rmc}{\mathrm{c}}

\newcommand{\rmf}{\mathrm{f}}

\newcommand{\rmo}{\mathrm{o}}

\newcommand{\rmr}{\mathrm{r}}

\newcommand{\rmH}{\mathrm{H}}

\newcommand{\rmK}{\mathrm{K}}

\newcommand{\rmQ}{\mathrm{Q}}

\newcommand{\calF}{\mathcal{F}}

\newcommand{\calX}{\mathcal{X}}
\newcommand{\calY}{\mathcal{Y}}

\newcommand{\fkc}{\mathfrak{c}}

\newcommand{\fkM}{\mathfrak{M}}

\newcommand{\mapright}[1]{%
\smash{\mathop{%
\hbox to 1cm{\rightarrowfill}}\limits^{#1}}}

\newcommand{\mapleft}[1]{%
\smash{\mathop{%
\hbox to 1cm{\leftarrowfill}}\limits_{#1}}}

\def\AGL{\operatorname{AGL}}

\def\height{\mathrm{ht}}
\def\Spec{\operatorname{Spec}}

\def\Syz{\mathrm{Syz}}

\def\gr{\mbox{\rm gr}}

\title[How many ideals whose quotient rings are Gorenstein exist?]{How many ideals whose quotient rings are Gorenstein exist?}

\author[Naoki Endo]{Naoki Endo}
\address{School of Political Science and Economics, Meiji University, 1-9-1 Eifuku, Suginami-ku, Tokyo 168-8555, Japan}
\email{endo@meiji.ac.jp}
\urladdr{https://www.isc.meiji.ac.jp/~endo/}

\thanks{2020 {\em Mathematics Subject Classification.} 13H10, 13A15, 13A02.}
\thanks{{\em Key words and phrases.} Gorenstein ring, numerical semigroup ring, $\rma$-invariant}
\thanks{The author was partially supported by JSPS Grant-in-Aid for Young Scientists 20K14299 and JSPS Grant-in-Aid for Scientific Research (C) 23K03058.}

\begin{document}

\maketitle

\setlength{\baselineskip} {15pt}

\begin{abstract}
For an Ulrich ideal in a Gorenstein local ring, the quotient ring is again Gorenstein. Aiming to further develop the theory of Ulrich ideals, this paper investigates a naive question of how many non-principal ideals whose quotient rings are Gorenstein exist in a given Gorenstein ring. The main result provides that the number of such graded ideals in a symmetric numerical semigroup ring $R$ coincides with the conductor of the semigroup. We furthermore provide a complete list of non-principal graded ideals $I$ in $R$ whose quotient rings $R/I$ are Gorenstein. 
\end{abstract}



\section{Introduction}\label{sec1}

We investigate a question of how many non-principal ideals whose residue class rings are Gorenstein exist in Gorenstein rings. 
The motivation of this query comes from the desire to develop the study of {\it Ulrich} ideals. The notion of Ulrich ideals is one of the modifications of that of {\it stable} maximal ideals introduced  in 1971 by his monumental paper \cite{L} of J. Lipman. The present one was formulated by S. Goto, K. Ozeki, R. Takahashi, K.-i. Watanabe, and K.-i. Yoshida \cite{GOTWY} in 2014, where the authors developed and consolidated the basic theory of Ulrich ideals. 
For example, if $A$ is a Gorenstein local ring and that for Ulrich ideals $I$ and $J$, the equality $I=J$ holds, provided  $\Syz_A^i(A/I) \cong \Syz_A^i(A/J)$ for some $i \ge 0$, where $\Syz_A^i(-)$ denotes the {\it i}-th syzygy module in the minimal free resolution. This shows that 
all the Cohen-Macaulay local rings {\it of finite CM-representation type}, i.e., there exist only finitely many non-isomorphic indecomposable maximal Cohen-Macaulay modules, contain only finitely many Ulrich ideals (\cite[Theorem 7.8]{GOTWY}, see also \cite[Theorem 5.1]{EG2}). In particular, by using the techniques from the representation theory of maximal Cohen-Macaulay modules, they succeeded in determining all the Ulrich ideals in Gorenstein local rings of finite CM-representation type of dimension at most $2$ (\cite[Theorem 9.5]{GOTWY}, \cite[Corollary 5.8]{GOTWY2}, see also \cite[Theorem 6.1]{GIT}). 
In addition, over a Gorenstein local ring $(A, \m)$, if $I$ and $J$ are Ulrich ideals of $A$ with $\m J \subseteq  I \subsetneq  J$, then $A$ must be a hypersurface (\cite[Corollary 7.6]{GOTWY}). Thus the ubiquity of Ulrich ideals reflects the singularities of base rings. 
Subsequently, the authors of \cite{GTT2} studied the structure of the complex $\RHom_A(A/I, \, A)$ in the derived category of a Cohen-Macaulay local ring $A$ with $d= \dim A$, and proved that  $A$ is Gorenstein if and only if $A/I$ is Gorenstein and $\mu_A(I) = d+1$, if an Ulrich ideal $I$ exists (\cite[Corollary 2.6 (b)]{GOTWY}, \cite[Corollary 2.6]{GTT2}), 
where $\mu_A(-)$ denotes the number of generators. 

Motivated by this observation, in this paper we investigate the following question.

\begin{ques}\label{q1}
Let $A$ be a Gorenstien ring with $d= \dim A>0$. 
How many ideals $I$ of $A$ with $\height_AI=1$ which satisfy the ring $A/I$ is Gorenstein and $\mu_A(I) \ge 2$ exist?
\end{ques}

Every Ulrich ideal in a one-dimensional Gorenstein local ring satisfies the conditions stated as in Queistion \ref{q1}. Whereas, based on the preceding researches on Ulrich ideals (\cite{EG, EGIM, GIT, GTT2}),
it is rather difficult to make a list of all the Ulrich ideals even for one-dimensional Cohen-Macaulay local rings, especially for numerical semigroup rings; see e.g., \cite[Theorem 3.9, Theorem 4.1]{EG}. 
In light of the result that there are only finitely many Ulrich ideals generated by monomials in numerical semigroup rings (\cite[Theorem 6.1]{GOTWY}), we start our investigation on Question \ref{q1} by going over graded ideals. Still, it remains unclear the question even if we restrict to graded ideals in numerical semigroup rings, which we will clarify in this paper. 


%

To state our result, let us explain the notation. 
For a numerical semigroup $H$, the  ring  
$$
R=k[H] = k[t^h \mid h \in H] \subseteq k[t]
$$
is called the {\it numerical semigroup ring} of $H$ over a field $k$, where $t$ denotes an indeterminate over $k$. Note that $R=k[H]$ is a one-dimensional Noetherian graded integral domain; moreover the ring $R$ enjoys a beautiful relation with its corresponding semigroup $H$. A typical example is that the maximum integer $\rmf(H)$ in the set $\bbN \setminus H$ coincides with the $\rma$-invariant $\rma(R)$ of the ring $R=k[H]$ (\cite[Example (2.1.9)]{GW}). Besides, the semigroup $H$ is {\it symmetric}, i.e., the equality $\#\{n \in H \mid n< \rmc(H) \} = \#(\bbN \setminus H)$ holds, if and only if its semigroup ring $R=k[H]$ is Gorenstein, where $\bbN$ denotes the set of non-negative integers, $\#(-)$ is the cardinality of a set, and $\rmc(H) = \rmf(H) + 1$ is the conductor of $H$. See \cite[Proposition 2.21]{HK2} or \cite[Theorem]{Kunz} for the proof of this fact.

With this notation, the main result of this paper is stated as follows. 

\begin{thm}\label{main}
Suppose that $R=k[H]$ is a Gorenstein ring. Then the equality
$$
\#\left\{I \mid I ~\text{is a graded ideal of} ~R ~\text{such that}~R/I~\text{is Gorenstein and}~\mu_R(I) \ge 2\right\}  =  \rmc(H)
$$
holds.
\end{thm}

Let us now explain how this paper is organized. To show Theorem \ref{main}, we need several auxiliaries which we will prepare in Section 2.  We actually provide them in a bit more general setting, not only for numerical semigroup rings. 
We shall prove Theorem \ref{main} in Section 3 starting with the case where $\rma(R/I) < \rma(R)$. In Section 4 we finally provide a complete list of non-principal graded ideals $I$ in $R$ whose quotient rings $R/I$ are Gorenstein. As an application of Theorem \ref{main}, we consider such ideals in the associated graded ring with respect to a certain filtration of ideals. Examples are explored as well.


\section{Preliminaries}


Let $R = \bigoplus_{n \ge 0}R_n$ be a one-dimensional Noetherian graded integral domain. Throughout this section, we assume $k=R_0$ is a field, and $R_n \ne (0)$ and $R_{n+1} \ne (0)$ for some $n \ge 0$. 
Let $W$ be the set of non-zero homogeneous elements in $R$. Note that the localization $W^{-1}R = K[t, t^{-1}]$ of $R$ with respect to $W$ is a {\it simple} graded ring, i.e., every non-zero homogeneous element is invertible, where $t$ is a homogeneous element of degree $1$ which is transcendental over $k$, and $K=[\, W^{-1}R\, ]_0$ is a field. There is an exact sequence
$$
0 \to R \to K[t, t^{-1}] \to \rmH^1_\m(R) \to 0
$$
of graded $R$-modules, where $\m$ denotes the graded maximal ideal of $R$ and $\rmH^1_\m(R)$ is the $1$st graded local cohomology module of $R$ with respect to $\m$. 
 As $R_0=k$ and $[\rmH^1_\m(R)]_0$ is a finite-dimensional $k$-vector space (remember that $\rmH^1_\m(R)$ is an Artinian $R$-module), the field extension $K/k$ is finite. Hence $k=K$, if $k$ is an algebraically closed field. 
Let $\overline{R}$ be the integral closure of $R$ in its quotient field $\rmQ(R)$. 

Recall that an integral domain $A$ is called {\it $N$-$1$} if the integral closure of $A$ in $\rmQ(A)$ is a finite $A$-module; and {\it $N$-$2$} if, for any finite field extension $L$ of $\rmQ(A)$, the integral closure of $A$ in $L$ is a finite $A$-module. We say that a Noetherian ring $B$ is {\it Nagata} if $B/\p$ is $N$-$2$ for every $\p \in \Spec B$ (\cite[(31.A) Definitions]{Matsumura}). Note that every field is Nagata, and every finitely generated algebra over a Nagata ring is Nagata (\cite[(31.H) Theorem]{Matsumura}). Thus every finitely generated algebra over a field which is an integral domain is a Nagata domain, so it is $N$-$1$. The reader may consult \cite[Section 31]{Matsumura} for the details. 

We begin with the following which was pointed out by S. Goto. 

\begin{lem}
The equality $\overline{R} = K[t]$ holds in $\rmQ(R)$. 
\end{lem}

\begin{proof}
Note that $\overline{R}$ is a graded ring and $\overline{R} \subseteq W^{-1}R=K[t, t^{-1}]$; see e.g., \cite[page 157]{ZS}. As the field $k$ is Nagata, so is the finitely generated $k$-algebra $R$. Thus $\overline{R}$ is a finite $R$-module. As $R_n = (0)$ for all $n<0$ and $R_0 = k$, we see that $[\,\overline{R}\,]_n = (0)$ for all $n<0$, $L=[\,\overline{R}\,]_0$ is a field, and $k \subseteq L \subseteq K$. Set $N = \bigoplus_{n > 0}[\,\overline{R}\,]_n$. Since the local ring $\overline{R}_N$ of $\overline{R}$ at the maximal ideal $N$ is a DVR, the ideal $N$ is principal. We choose a homogeneous element $f \in \overline{R}$ of degree $q>0$ such that $N = f\overline{R}$. Hence $\overline{R} = L[N] = L[f] \subseteq W^{-1}R =K[t, t^{-1}]$. 
Besides, because $\overline{R}[f^{-1}] = L[f, f^{-1}]$ is a simple graded ring and $R \subseteq \overline{R}[f^{-1}]$, we have $W^{-1}R \subseteq \overline{R}[f^{-1}]=L[f, f^{-1}]$. Therefore 
$$
K[t, t^{-1}] = L[f, f^{-1}]
$$
so that $K=L$ and $q = 1$. This shows $\overline{R} = L[f] = K[f] = K[t]$, as claimed. 
\end{proof}


For $R$-submodules $X$ and $Y$ of $\rmQ(R)$, let $X:Y = \{a \in \rmQ(R) \mid aY \subseteq X\}$. If we consider ideals $I, J$ of $R$, we set $I:_RJ =\{a \in R \mid aJ \subseteq I\}$. Hence $I:_RJ = (I:J) \cap R$.

\begin{rem}
Let $I$ be a non-zero graded ideal of $R$. It is straightforward to check that $R:I$ is a graded $R$-submodule of $K[t, t^{-1}]$ which contains $R$. In addition, the natural isomorphism $R:I  \overset{\cong}{\longrightarrow} \Hom_R(I, R), \, \alpha \mapsto (x \mapsto \alpha x)$ is graded. Thus, provided $k=K$, every homogeneous component of $\Hom_R(I, R)$ has dimension, as a $k$-vector space, at most $1$. 
\end{rem}

For a Cohen-Macaulay graded ring $A=\bigoplus_{n\ge 0}A_n$ such that $A_0$ is a local ring, we set $\rma (A)=\operatorname{max} \{n \in \Z \mid [\rmH^d_\fkM({A})]_n \ne (0) \}$ which is called the {\it $\rma$-invariant} of $A$ (\cite[Definition (3.1.4)]{GW}). Here, $\fkM$ denotes the unique graded maximal ideal of $A$, $d=\dim A$, and $\{[\rmH_\fkM^d(A)]_n\}_{n \in \Bbb Z}$ is the homogeneous components of the $d$-th graded local cohomology module $\rmH_\fkM^d(A)$ of $A$ with respect to $\fkM$. 
When $A$ admits the graded canonical module $\rmK_A$, one has $\rma(A) = -\operatorname{min} \{ n \in \Z \mid [\K_A]_n \ne (0) \}$.

Let $M$ be a graded $R$-module and $\ell$ an integer. 
Let $M(\ell)$ denote the graded $R$-module whose underlying $R$-module is the same as that of the $R$-module $M$ and the grading is given by $[M(\ell)]_n = M_{\ell + n}$ for all $n \in \Bbb Z$. When $M$ is finitely generated, we denote by $\mu_R(M)$ the minimal number of generators of $M$. 

With this notation, we furthermore assume $R$ admits a graded canonical module $\rmK_R$. 
Let $(-)^{\vee} = \Hom_R(-, \rmK_R)$ denote the canonical dual functor.  
We then have the following.

\begin{lem}\label{1.5}
Suppose that $R$ is a Gorenstein ring. 
Let $I$ be a graded ideal of $R$ such that $R/I$ is Gorenstein and $\mu_R(I) \ge 2$. Then the following assertions hold true. 
\begin{enumerate}[$(1)$]
\item $[I^{\vee}]_{-\rma(R)} \ne (0)$ and $[I^{\vee}]_{-\rma(R/I)} \ne (0)$.
\item $\min\left\{ n \in \bbZ \mid \left[I^{\vee}\right]_n \ne (0)\right\} = \min\{-\rma(R), -\rma(R/I)\}$. 
\item $\mu_R(I^{\vee}) = 2$. 
\item If $k=K$, then $\rma(R) \ne \rma(R/I)$.
\item $I^{\vee} = R f + R g$ for some $f \in [I^{\vee}]_{-\rma(R)}$ and $g \in [I^{\vee}]_{-\rma(R/I)}$. 
\end{enumerate}
\end{lem}

\begin{proof}
We set $a = \rma(R)$, $b = \rma(R/I)$, and $n= \min\left\{ n \in \bbZ \mid \left[I^{\vee}\right]_n \ne (0)\right\}$. 
By taking the functor $(-)^{\vee}$ to the exact sequence $0 \to I \to R \to R/I \to 0$, we get the sequence
$$
(*) \quad\quad\quad\quad\quad\quad0 \to R(a) \to I^{\vee} \to (R/I)(b) \to 0 \quad\quad\quad\quad\quad\quad \quad
$$
of graded $R$-modules, because $\rmK_R \cong R(a)$ and $\Ext^1_R(R/I, \rmK_R) \cong \rmK_{(R/I)} \cong (R/I)(b)$. This shows $[I^{\vee}]_{-a} \ne (0)$, $[I^{\vee}]_{-b} \ne (0)$, and $n= \min\{-a, -b\}$. Besides, the exact sequence $(*)$ implies $\mu_R(I^{\vee}) \le 2$. As $I^{\vee \vee} \cong I$ and $\mu_R(I) \ge 2$,  we get $\mu_R(I^{\vee}) = 2$. This proves the assertions $(1)$, $(2)$, and $(3)$. 

If $k=K$, then all the homogeneous components of $I^{\vee} \cong \Hom_R(I, R)(a)$, as a $k$-vector space, have dimension at most $1$. Thus $a \ne b$, and the assertion $(4)$ holds. 

Recall that $\m$ is the graded maximal ideal of $R$. 
By applying the functor $R/\m \otimes_R -$ to the sequence $(*)$, we have the exact sequence of the form:
$$
(R/\m)(a) \overset{\xi}{\to} I^{\vee}/\m I^{\vee} \overset{\eta}{\to} (R/\m)(b) \to 0.
$$
As $\mu_R(I^{\vee}) = 2$, the map $\xi$ is injective. We choose $f \in [I^{\vee}]_{-a}$ and $g \in [I^{\vee}]_{-b}$ such that $\overline{f} = \xi(1)$ and $\eta(\overline{g}) = 1$, where $\overline{*}$ denotes the image in $I^{\vee}/\m I^{\vee}$. 
Then the images of $f, g$ form a $k$-basis of $I^{\vee}/\m I^{\vee}$. Hence $I^{\vee} = R f + R g$ by Nakayama's lemma.
\end{proof}

\begin{rem}
If $R$ is a numerical semigroup ring over a field $k$, then $k=K$. Whereas, if $k=K$, e.g., $k$ is an algebraically closed field, then the ring $R$ is isomorphic to a semigroup ring of a numerical semigroup (\cite[Proposition (2.2.11)]{GW}).
\end{rem}

\section{Proof of Theorem \ref{main}}

We first fix the notation on which all the results in this section are based.

\begin{setup}
Let $\Bbb N$ be the set of non-negative integers and  $a_1, a_2, \ldots, a_{\ell} \in \bbZ~(\ell \ge 1)$ be positive integers such that $\gcd (a_1, a_2, \ldots, a_{\ell}) = 1$. We set
$$
H = \left<a_1, a_2, \ldots, a_\ell\right>=\left\{\sum_{i=1}^\ell c_ia_i  ~\middle|~   c_i \in \Bbb N~\text{for~all}~1 \le i \le \ell \right\}
$$
and call it the {\it numerical semigroup} generated by $\{a_i\}_{1 \le i \le \ell}$. The reader may consult the book \cite{RG} for the fundamental results on numerical semigroups. 
Let $S = k[t]$ denote the polynomial ring over a field $k$, and define 
$$
k[H] = k[t^{a_1}, t^{a_2}, \ldots, t^{a_\ell}] \subseteq S
$$
which we call the {\it semigroup ring} of $H$ over $k$. 
The ring $R=k[H]$ forms a Noetherian integral domain with $\dim R=1$ and is a $\Bbb Z$-graded subring of $S$ whose grading $\{R_n\}_{n \in \Bbb Z}$ is given by 
$$
R_n=
\begin{cases}
\ kt^n & \text{if} \ \ n \in H, \\
\ (0) & \text{otherwise}.
\end{cases}
$$
In addition, $S$ is a birational module-finite extension of $R$, so that  $\overline{R} = S$, where $\overline{R}$ denotes the integral closure of $R$ in its quotient field $\rmQ(R)$.
Let 
$$
\rmc(H) = \min \{n \in \Bbb Z \mid m \in H~\text{for~all}~m \in \Bbb Z~\text{such~that~}m \ge n\},
$$
and set $\rmf(H) = \max ~(\Bbb Z \setminus H)$ which is called the {\it Frobenius number} of $H$.
By \cite{GW}, we get
\begin{center}
$R:S = t^{\rmc(H)}S$ \  and \ $\rmf(H) =  \rmc(H) -1 = \rma(R)$.
\end{center}
Note that, for each non-zero ideal $I$ in $R$, we have $\rma(R/I) \in H$, whence $\rma(R/I) \ne \rma(R)$; see also Lemma \ref{1.5} (4). We set $a = \rma(R)$ and $c = \rmc(H)$. 
\end{setup}


\begin{rem}\label{intcld}
If $R$ is integrally closed, then $c=0$ and there is no non-principal ideals in $R$. 
\end{rem}

The following plays a key in our argument. 

\begin{prop}\label{key}
Suppose that $R=k[H]$ is a Gorenstein ring. Then the equality
$$
\#\left\{I \mid I ~\text{is a graded ideal of} ~R ~\text{such that}~ \rma(R/I) < a~\text{and}~\mu_R(I) \ge 2\right\} \, = \, \frac{c}{2} 
$$
holds.
\end{prop}

\begin{proof}
Let $\calY_R$ be the set of graded ideals $I$ of $R$ such that $R/I$ is Gorenstein, $\rma(R/I) < a$, and $\mu_R(I) \ge 2$. By Remark \ref{intcld}, we may assume $R \ne \overline{R}$. Thus $\calY_R \ne \emptyset$. 
For each $I \in \calY_R$, since $\rma(R/I) < a$, we then have $R_m \subseteq I$ for all $m \ge c = a+1$. By setting $J = R:I$, we see that $J$ is a graded ideal of $R$ and  
$$
R \subseteq J \subseteq R:\fkc = R:(R:\overline{R}) = \overline{R}
$$
where the second inclusion follows from $\fkc \subseteq I$ and the last equality holds by \cite[Bemerkung 2.5]{HK} (remember that $R$ is a Gorenstein ring). This shows $J = (1,\, t^m)$ for some $m \in \Bbb N \setminus H$. So we can consider the map 
$$
\Phi : \calY_R \to {\Bbb N} \setminus H
$$ 
defined by $\Phi(I) = m$ for each $I \in \calY_R$, where $R : I = (1,\, t^m)$. 

Conversely, for each $m \in  {\Bbb N} \setminus H$, we set $J = (1,\, t^m)$. Then $R \subsetneq J \subseteq \overline{R} =k[t]$. By setting $I = R:J$, we have
$$
\fkc = R:\overline{R} \subseteq R:J = I \subsetneq R
$$
which yield that $\rma(R/I) < a$, $\mu_R(I) \ge 2$, and the ring $R/I$ is Gorenstein. Indeed, since $t^c \overline{R} = \fkc \subseteq I$ and $\rma(R/I) \in H$, we have $\rma(R/I) < a$. The $\rmK_R$-dual $(-)^{\vee}$ of the exact sequence $0 \to I \to R \to R/I \to 0$ induces the sequence
$$
0 \to R(a) \overset{\varphi}{\to} I^{\vee} \to \Ext^1_R(R/I, \rmK_R) \to 0
$$
of graded $R$-modules, because $\rmK_R \cong R(a)$. Let $f = \varphi(1)$. Then $f \in [I^{\vee}]_{-a}$ forms a part of a minimal basis of $I^{\vee}$.  
As $I^{\vee} \cong J$ and $\mu_R(J) = 2$, we get $\mu_R(I) \ge 2$; while the $R$-module $\Ext^1_R(R/I, \rmK_R)$ is cyclic, so that $R/I$ is a Gorenstein ring, because $\Ext^1_R(R/I, \rmK_R) \cong \rmK_{(R/I)}$ is the canonical module of $R/I$. Here, the proof of above especially shows that if $\rma(R/I)<a$ then $R/I$ is Gorenstein. 
Hence, we define the map
$$
\Psi : {\Bbb N} \setminus H \to \calY_R, \ \ m \mapsto R:(1,\, t^m)   
$$
and it is straightforward to check the composite maps $\Phi \circ \Psi$ and $\Psi \circ \Phi$ are identity. In particular, the map $\Phi$ is bijective. Therefore
$$
\#\left\{I \mid I ~\text{is a graded ideal of} ~R, \, \rma(R/I) < a, \, \text{and} \ \mu_R(I) \ge 2\right\} = \#\calY_R = \#({\Bbb N} \setminus H) = \frac{c}{2} 
$$
where the last equality follows from the fact that $H$ is symmetric, i.e., $R$ is Gorenstein. This completes the proof. 
\end{proof}

We are ready to prove Theorem \ref{main}.

\begin{proof}[Proof of Theorem \ref{main}]
Let $\calX_R$ be the set of graded ideals $I$ of $R$ such that $R/I$ is Gorenstein and $\mu_R(I) \ge 2$. By Remark \ref{intcld}, we may assume $R \ne \overline{R}$. So $\calX_R \ne \emptyset$. 
By Proposition \ref{key}, it suffices to show that the number of ideals $I \in \calX_R$ with $\rma(R/I) > a$ is a half of the conductor $c$ of $H$.

For each $I \in \calX_R$, we have $\mu_R(I^{\vee}) = 2$, so we can write
\begin{center}
$I^{\vee} = R f + R g$ \ \ for some \! $f \in [I^{\vee}]_{-a}$ and $g \in [I^{\vee}]_{-\rma(R/I)}$.
\end{center}
Set $b= \rma(R/I)$. 
Since $(0):_Rg =(0)$, we get the exact sequence
$$
0 \to R(b) \overset{\xi}{\to} I^{\vee} \to C \to 0
$$
of graded $R$-modules, where $\xi(1) = g$ and $C = \Coker \xi$. We consider the graded ideal $J = (0):_R C$ of $R$. As $C$ has a finite length and $C \cong I^{\vee}/Rg \ne (0)$, we get $(0) \ne J \subsetneq R$. Besides, we have the isomorphism 
$$
C \cong R \overline{f} \cong (R/J)(a)
$$  
as a graded $R$-module, where $\overline{*}$ denotes the image in $I^{\vee}/Rg$. 
Hence we obtain the sequence
$$
0 \to R(b) \to I^{\vee} \to (R/J)(a) \to 0
$$
of graded $R$-modules. By applying the functor $(-)^{\vee}$ to the above sequence, we have 
$$
0 \to I \to R(a-b) \to \Ext^1_R(R/J, \rmK_R)(-a) \to 0
$$
because $I^{\vee \vee} \cong I$ and $R(b)^{\vee} = \Hom_R(R(b), \rmK_R) \cong \Hom_R(R(b), R(a)) \cong R(a-b)$. In particular, $\rmK_{(R/J)} \cong \Ext^1_R(R/J, \rmK_R)$ is cyclic; hence $R/J$ is Gorenstein. By letting $\alpha = \rma(R/J)$, we have $\rmK_{(R/J)} \cong (R/J)(\alpha)$. Therefore, by changing the shift by $b-a$, we get the exact sequence
$$
0 \to I(b-a) \overset{\psi}{\to} R \to (R/J)(\alpha -a + b-a) \to 0
$$
of graded $R$-modules. The degree $0$ part of the following isomorphism
$$
R/\Im \psi \cong (R/J)(\alpha -2a + b)
$$
induces $\alpha -2a + b = 0$; while $I(b-a) \cong \Im \psi \cong J$. 
Hence, $\rma(R/J) = \alpha = 2a -b$ and $I \cong J(a-b)$ as a graded $R$-module. In particular, $\mu_R(J) \ge 2$. Thus $J \in \calX_R$.

Let $W$ be the set of non-zero homogeneous elements in $R$. Consider the simple graded ring $W^{-1}R = k[t, t^{-1}]$, where $t$ is a homogeneous element of degree $1$ which is transcendental over $k$.
Hence we have the commutative diagram below:
$$
\xymatrix{
k[t, t^{-1}](a-b) = W^{-1}(J(a-b)) \ar[r]^{\ \ \ \ \  \ \ \ \ \ \ \  \cong}  & W^{-1}I = k[t, t^{-1}]\\
J(a-b) \ar[r]^{\cong} \ar[u] & I = t^{b-a}J\ar[u]
}
$$
Note that the induced isomorphism $k[t, t^{-1}](a-b) \overset{\cong}{\longrightarrow} k[t, t^{-1}]$ is given by the homothety of homogeneous element of degree $b-a$. Therefore $I = t^{b-a}J$. 


To sum up this argument, for each $I \in \calX_R$, there exists a graded ideal $J \in \calX_R$ satisfying 
\begin{center}
$\rma(R/J) = 2 a - \rma(R/I)$ \ \ and \  \ $I = t^{\rma(R/I)-a}J$.
\end{center}
This shows, if $\rma(R/I) > a$ (resp. $\rma(R/I) < a$), then $\rma(R/J) < a$ (resp. $\rma(R/J) > a$). 
So, there is a one-to-one correspondence between the set of ideals $I \in \calX_R$ with $\rma(R/I) > a$, and the set of ideals $J \in \calX_R$ with $\rma(R/J) < a$. Finally we conclude that 
$$
\# \calX_R  = \# \{I \in \calX_R\mid \rma(R/I)>a \} + \# \{I \in \calX_R\mid \rma(R/I)<a \} = \frac{c}{2} + \frac{c}{2} = c
$$ 
as desired. 
\end{proof}

\section{Corollaries and examples}

We summarize some consequences of Theorem \ref{main}. In this section we maintain the notation as in Setup 3.1. Let $\calX_R$ be the set of graded ideals $I$ of $R$ such that $R/I$ is Gorenstein and $\mu_R(I) \ge 2$. 
Recall that  $a = \rma(R)$ and $c = \rmc(H)$.

The direct consequence of the proof of Theorem \ref{main} is stated as follows, which is useful to compute concrete examples.

\begin{cor}\label{4.1}
Suppose that $R=k[H]$ is a Gorenstein ring. For each $I \in \calX_R$, we set 
$J=t^{a-\rma(R/I)}I$. Then the following assertions hold true.
\begin{enumerate}[$(1)$]
\item $J \in \calX_R$ and $\rma(R/J) = 2a - \rma(R/I)$. Hence, if $\rma(R/I) < a$ $($resp. $\rma(R/I) > a$$)$, then $\rma(R/J) > a$ $($resp. $\rma(R/J) < a$$)$.
\item $\rma(R/I) \in H$, $a \ne \rma(R/I)$, and $a-\rma(R/I) \in \Bbb Z \setminus H$. 
\item If $\rma(R/I) <a $, then $a-\rma(R/I) \in \Bbb N \setminus H$.
\item If $\rma(R/I) >a $, then $\rma(R/I)-a \in \Bbb N \setminus H$. 
\end{enumerate}
\end{cor}

\begin{proof}
We already proved the assertion $(1)$ in the proof of Theorem \ref{main}. Recall that $\rma(R/I) \in H$ and $a \not\in H$. So $a \ne \rma(R/I)$. 
As $H$ is symmetric and $\rma(R/I) \in H$, we see that $a-\rma(R/I) \in \Bbb Z \setminus H$. In particular, if $\rma(R/I) <a$, then $a-\rma(R/I) \in \Bbb N \setminus H$. On the other hand, we assume $\rma(R/I) >a$. Since $J \in \calX_R$ and $\rma(R/J) < a$, we conclude that $\rma(R/I) - a = a - \rma(R/J) \in \bbN \setminus H$, as claimed. 
\end{proof}

The next provides a complete list of graded ideals in $\calX_R$. 

\begin{cor}\label{4.2}
Suppose that $R=k[H]$ is a Gorenstein ring. Then the equality 
$$
\calX_R = \{R:_Rt^m, \, t^m(R:_Rt^m) \mid m \in \bbN \setminus H\}
$$ 
holds. Moreover, for each $m \in \bbN \setminus H$, one has 
\begin{center}
$\rma(R/R:_Rt^m) = a-m$ \ and \ $\rma(R/t^m(R:_Rt^m)) = a+m$.
\end{center}
\end{cor}

\begin{proof}
Note that $R:(1, t^m) = R:_Rt^m$ for all $m \in \bbN \setminus H$. 
By Proposition \ref{key}, there is a one-to-one correspondence below:
\begin{center}
$\bbN \setminus H \ \longleftrightarrow \ \{I \in \calX_R \mid \rma(R/I) < a\}$, \ $m \longmapsto R:(1, t^m)$.
\end{center}
This shows the equality $\{I \in \calX_R \mid \rma(R/I) < a\} = \{R:_Rt^m \mid m \in \bbN \setminus H\}$. Besides, the proof of Theorem \ref{main} guarantees that the map 
\begin{center}
$\{I \in \calX_R \mid \rma(R/I) < a\} \ \longleftrightarrow \ \{I \in \calX_R \mid \rma(R/I) > a\}$, \ $I \longmapsto t^{a-\rma(R/I)}I$
\end{center}
is bijective. Hence $\{I \in \calX_R \mid \rma(R/I) > a\} = \{t^{a-\rma(R/{R:_Rt^m})}(R:_Rt^m) \mid m \in \bbN \setminus H\}$ holds. Since $H$ is symmetric and $c=a+1$, it is straightforward to check that $\rma(R/R:_Rt^m) = a - m$ for all $m \in \bbN \setminus H$. Therefore the equality
$$
\calX_R = \{R:_Rt^m, \, t^m(R:_Rt^m) \mid m \in \bbN \setminus H\}
$$ 
holds. Furthermore, by Corollary \ref{4.1} (1), we have the equalities 
$$
\rma(R/t^m(R:_Rt^m)) = 2a - \rma(R/R:_Rt^m) = 2a - (a-m) = a + m
$$
which complete the proof. 
\end{proof}

The ideals of the forms $R:_Rt^m$ and $t^m(R:_Rt^m)$ are easy to compute,  especially in numerical semigroup rings, and provide numerous examples illustrating Theorem \ref{main}.

\begin{ex}
Let $k[t]$ be the polynomial ring over a field $k$ and $R = k[H]$ the semigroup ring of a numerical semigroup $H$. Then the following assertions hold. 
\begin{enumerate}[$(1)$]
\item Let $H = \left<2, 2\ell + 1\right>~(\ell \ge 1)$. Then $\rmc(H) = 2\ell$ and the equality 
$$
\calX_R = \{(t^2, t^{2\ell + 1}), (t^4, t^{2\ell + 1}), \ldots, (t^{2\ell}, t^{2\ell + 1}), (t^{2\ell +1}, t^{4 \ell}), (t^{2\ell +1}, t^{4 \ell-2}), \ldots, (t^{2\ell +1}, t^{2 \ell + 2})\}
$$ holds. 
\item Let $H = \left<3, 4\right>$. Then $\rmc(H) = 6$ and the equality 
$$
\calX_R = \{(t^3, t^4), (t^4, t^6), (t^3, t^8), (t^8, t^9), (t^6, t^8), (t^4, t^9)\}
$$ holds. 
\item Let $H = \left<3, 5\right>$. Then $\rmc(H) = 8$ and the equality 
$$
\calX_R = \{(t^3, t^5), (t^5, t^6), (t^3, t^{10}), (t^5, t^9), (t^{10}, t^{12}), (t^9, t^{10}), (t^5, t^{12}), (t^6, t^{10})\}
$$ holds. 
\item Let $H = \left<n, n+1, \ldots, 2n-2\right>~(n \ge 4)$. Then $\rmc(H) = 2n$ and the equality 
\begin{eqnarray*}
\calX_R \!\! &=& \!\! \{(t^n, t^{n+1}, \ldots, t^{2n-2}), (t^{n+1}, t^{n+2}, \ldots, t^{2n-2}, t^{2n})\} \\
&\bigcup& \!\!  \{(t^n, t^{n+1}, \ldots, t^{n+i-1}, t^{n+i+1}, \ldots, t^{2n-2}) \mid 1 \le i \le n-2\} \\
&\bigcup& \!\! \{(t^{3n-1}, t^{3n}, \ldots, t^{4n-3}), (t^{2n}, t^{2n+1}, \ldots, t^{3n-3}, t^{3n-1})\} \\
&\bigcup& \!\! \{(t^{2n-i-1}, t^{2n-i}, \ldots, t^{2n-2}, t^{2n}, \ldots, t^{3n-i-3}) \mid 1 \le i \le n-2\} 
\end{eqnarray*}
holds. 
\end{enumerate}
\end{ex}

Let $H_1 = \left<a_1, a_2, \ldots, a_{\ell}\right>$ and $H_2 = \left<b_1, b_2, \ldots, b_{m}\right>~(\ell, m \ge 1)$ be numerical semigroups. We choose $d_1 \in H_2 \setminus \{b_1, b_2, \ldots, b_{m}\}$ and $d_2 \in H_1 \setminus \{a_1, a_2, \ldots, a_{\ell}\}$ such that $\gcd(d_1, d_2) = 1$. 
We say that 
$$
H=\left<d_1H_1, d_2 H_2\right> = \left<d_1a_1, d_1a_2, \ldots, d_1a_{\ell}, d_2b_1, d_2b_2, \ldots, d_2b_m\right>
$$
is a {\it gluing} of $H_1$ and $H_2$ with respect to $d_1 \in H_2$ and $d_2 \in H_1$.  

Note that every three-generated symmetric numerical semigroup $H$ is obtained by gluing of a two-generated numerical semigroup $H_1$ and $\Bbb N$ (\cite[Section 3]{Herzog}, \cite[Proposition 3]{W}). 
Let $a, b \in \Bbb Z$ be positive integers with $\gcd(a, b) = 1$. We set $H_1 = \left<a, b\right>$ and assume that $H_1$ is minimally generated by two-elements. Choose $c \in H_1$ and $d \in \Bbb  N$ so that $c, d$ satisfy  the conditions that $c > 0$, $d > 1$, $c \not\in \{a, b\}$, and $\gcd(c, d) = 1$.  Hence,  $\gcd(da, db, c) = 1$. 
We consider a gluing 
$
H = \left<dH_1, c\bbN \right>
$
of $H_1$ and $\bbN$ with respect to $d \in \bbN$ and $c \in H_1$. 
Let $k$ be a field. 
We then have the isomorphism
$$
k[H] \cong k[X, Y, Z]/(X^{b} - Y^{a}, Z^{d} - X^{m}Y^n)
$$
of $k$-algebras, where $c = am + bn$ with $m, n \in \bbN$. 
Hence, $\rma(k[H]) = d(ab-a-b) + (d-1)c$.

\begin{cor}
Let $H$ be a three-generated symmetric numerical semigroup. Under the same notation of above, the equality
$$\# \calX_{k[H]} = d(ab-a-b) + (d-1)c + 1$$ holds. 
\end{cor}

\begin{ex}
Let $k$ be a field and $H = \left< 4, 6, 7\right>$. Then $H = \left<2 \left<2, 3\right>, 7 \bbN \right>$ and 
$$
R=k[H] \cong k[X, Y, Z]/(X^3 - Y^2, Z^2 - X^2Y).
$$
In particular, $\rma(R) = 9$ and $\#\calX_R = \rmc(H) =10$. Indeed, we have the equality
\begin{eqnarray*}
\calX_R \!\! &=& \!\! \{ (t^4, t^6, t^7), (t^6, t^7, t^8), (t^4, t^7), (t^4, t^6), (t^6, t^7)\} \\
\!\! &\bigcup& \!\! \{(t^{13}, t^{15}, t^{16}), (t^{11}, t^{12}, t^{13}), (t^7, t^{10}), (t^6, t^8), (t^7, t^8)\}.
\end{eqnarray*}
\end{ex}

For an $R$-module $M$, we denote by $\left[M\right]$ the isomorphism class of $M$.

\begin{cor}
Suppose that $R=k[H]$ is a Gorenstein ring. Then the equalities
\begin{center}
$\left\{[\, I\, ] \mid I \in \calX_R \right\} = \{[R:_Rt^m] \mid m \in \bbN \setminus H\}$ \ and \ $\#\left\{[\, I\, ] \mid I \in \calX_R \right\} = \dfrac{c}{2}$ 
\end{center}
hold.
\end{cor}

\begin{proof} 
Note that $R:_Rt^m \cong t^m(R:_Rt^m)$ as an $R$-module for each $m \in \bbN \setminus H$. This shows the equality $\left\{[\, I\, ] \mid I \in \calX_R \right\} = \{[R:_Rt^m] \mid m \in \bbN \setminus H\}$ and its cardinality is at most a half of $c$. We now assume $R:_Rt^m = t^{\ell}(R:_Rt^{m'})$ for some $m, m' \in \bbN \setminus H$ and $\ell \in \Bbb Z$. Then
$$
R:(1, t^m) = t^{\ell}(R:(1, t^{m'})) = R:(t^{-\ell} (1, t^{m'})) = R:  (t^{-\ell}, t^{-\ell + m'}).
$$
As $R$ is Gorenstein, we have $(1, t^m) = (t^{-\ell}, t^{-\ell + m'})$ in $k[t, t^{-1}]$. Since $-\ell < -\ell + m'$, we have $\ell = 0$ and $m = -\ell + m'=m'$. Hence the cardinality of $\left\{[\, I\, ] \mid I \in \calX_R \right\}$ is the half of $c$. 
\end{proof}

\begin{rem}
There exists a one-dimensional local Gorenstein numerical semigroup ring $A$ with infinite residue class field (e.g., ${\Bbb Q}[[t^3, t^7]]$, ${\Bbb C}[[t^4, t^5, t^6]]$) admitting infinitely many two-generated Ulrich ideals. Hence $\calX_A = \infty$. 
\end{rem}

When $A$ is a local ring, although the set $\calX_A$ is not necessarily finite, there is an associated graded ring $G$ with respect to a filtration of ideals such that $\calX_G$ is a finite set.


Let $(A, \m)$ be a Noetherian local ring with $\dim A=1$ and $V=\overline{A}$ the integral closure of $A$ in its total ring $\rmQ(A)$ of fractions. Assume that $V$ is a DVR which is a module-finite extension of $A$ and $A/\m \cong V/\n$, where $\n = tV~(t\in V)$ denotes the maximal ideal of $V$. 
Let $\rmo(-)$ denote the $\n$-adic valuation (or the order function) of $V$ and set
$$
v(A)= \{\rmo(f) \mid 0 \ne f \in A\}.
$$
Then, $H_A=v(A)$ is called the {\it value semigroup} of $A$, which is indeed a numerical semigroup. Let $\fkc = A:V$ denote the conductor of $A$. Then $\fkc = t^{\rmc(H_A)}V$ and $\rmc(H_A) = \ell_A(V/\fkc)$. Note that $A$ is Gorenstein if and only if $H_A = v(A)$ is symmetric (\cite[Theorem]{Kunz}).  

For each $\ell \in \Bbb Z$, we set $F_{\ell} = \n^{\ell} \cap A$. Then $\calF = \{F_{\ell}\}_{\ell \in \Bbb Z}$ is a filtration of ideals in $A$. We define
$$
G=G(\calF) = \bigoplus_{\ell \ge 0} F_{\ell}/F_{\ell+1} = \bigoplus_{\ell \ge 0} (\n^{\ell} \cap A)/(\n^{\ell+1} \cap A)
$$
and call it the {\it associated graded ring of $A$ with respect to $\calF$}. Note that, for each $\ell \ge 0$, $G_{\ell} \ne (0)$ if and only if $\ell \in H_A$. This shows $H_A =\{\ell \ge 0 \mid G_{\ell} \ne (0)\}$ and the isomorphism below:
$$
G = \bigoplus_{\ell \ge 0} F_{\ell}/F_{\ell+1} \cong (A/\m)[H_A]. 
$$
With this notation we have the following. 

\begin{cor}
Let $(A, \m)$ be a one-dimensional Gorenstein complete local domain with algebraically closed residue class field. Let $G=G(\calF)$ be the associated graded ring of $A$ with respect to the filtration $\calF=\{\n^{\ell} \cap A\}_{\ell \in \Bbb Z}$, where $\n$ denotes the maximal ideal of $V=\overline{A}$. 
Then the equality
$$
\#\left\{I \mid I ~\text{is a graded ideal of} ~G ~\text{such that}~G/I~\text{is Gorenstein and}~\mu_G(I) \ge 2\right\} = \rmc(H_A)
$$
holds, where $H_A = v(A)$ denotes the value semigroup of $A$. 
\end{cor}





 
\begin{ac}
The author would like to thank the referee for his/her valuable comments and suggestions.
\end{ac}

\vspace{-1em}

\end{document}